\documentclass[10pt,reqno]{amsart}
\usepackage{amsmath}
\usepackage{amsfonts}
\usepackage{amssymb}
\usepackage{amsthm}
\usepackage{amscd}

\usepackage{a4wide}
\newcommand{\Section}[1]{\section{#1} \setcounter{equation}{0}}

\newtheorem{thm}{Theorem}
\newtheorem{lem}{Lemma}[section]

\theoremstyle{remark}
\newtheorem*{rem}{Remark}

\theoremstyle{definition}

\begin{document}
\large
\title{On %zeros of
a form of degree $d$ in $2d+1$ variables ($d\geq 4$)}
\author{Manoj Verma}
\address{MANOJ VERMA: Institute of Mathematical Sciences,
CIT Campus, Taramani, Chennai 600113, India}
\email{mverma@imsc.res.in}

\begin{abstract} For $k\geq 2$, we derive an asymptotic formula for the number of zeros of the forms
\begin{center}
$\prod_{i=1}^{k}(x_{2i-1}^2+x_{2i}^2)+\prod_{i=1}^{k}(x_{2k+2i-1}^2+x_{2k+2i}^2)-x_{4k+1}^{2k}$
\end{center}
and
\begin{center}
$x_1\prod_{i=1}^{k}(x_{2i}^2+x_{2i+1}^2)+x_{2k+2}\prod_{i=1}^{k}(x_{2k+2i+1}^2+x_{2k+2i+2}^2)-x_{4k+3}^{2k+1}$
\end{center} in the box $1\leq x_i\leq P$ using the circle method.
\end{abstract}
\maketitle
\Section {Introduction}
\label{introduction}
\noindent The problems of representation of positive integers and of non-trivial
representation of zero by forms with integral coefficients have attracted attention of many mathematicians since the beginning of the 20th century. Most of the work in this area has been carried out using the circle method. The square-root barrier implies that the circle method is unlikely to yield an asymptotic formula in case of a form  of degree $d$ with less than $2d+1$ variables. In 1962 Birch, Davenport and Lewis \cite{BDL} found an asymptotic formula for the number representations of zero in a box $|x_i|\leq P$ by the cubic forms in seven variables 
of the type
$$N_1(x_1,x_2,x_3)+N_2(x_4,x_5,x_6)+a_7x_7^3$$ where $N_1$ and $N_2$ are norm
forms of some cubic field extensions of $\mathbb Q$ and $a_7$ is a nonzero integer. Incidentally, whereas an asymptotic formula for the number of representations of an integer as a sum of eight cubes has been established by Vaughan \cite{Vaughan1}, an asymptotic formula in the case of a sum of seven cubes seems to be out of reach as of now. However, a lower bound of the expected order of magnitude has been established by Vaughan \cite{Vaughan3}. 
A few more families of cubic forms in seven variables have been dealt with. %In \cite{c7i}, \cite{c7ii} the author found asymptotic formulas for the number of representations of zero inside a box and of a large positive integer inside a suitable box by the cubic forms in seven variables that can be written as $L_1(x_1,x_2,x_3) Q_1(x_1,x_2,x_3)+ L_2(x_4,x_5,x_6) Q_2(x_4,x_5,x_6) + a_7 x_7^3$ where $L_1$ and $L_2$ are linear forms, $Q_1$ and $Q_2$ are quadratic forms, $L_1Q_1$ and $L_2Q_2$ are non-degenerate, and $a_7$ is a non-zero integer. 
However, for $d \geq 4$ no example of an asymptotic formula for the number of representations of either zero or of large positive integers by a form of degree $d$ in $2d+1$ variables seems to exist in the literature. In this paper, for each $d\geq 4$, we establish an asymptotic formula for the number of zeros of a particular form of degree $d$ in $2d+1$ variables in a suitable expanding box using the circle method. It turns out that our proof goes through in cases $d=1, 2, 3$ as well and retrieves previously known results in those cases. For $d$ even, say $d=2k$, $k\geq 1$, we look at the number zeros of the form
\begin{equation*} \label{e:one1e}
f({\bf x})=f(x_1, \ldots, x_{4k+1})=\prod_{i=1}^{k}(x_{2i-1}^2+x_{2i}^2)+\prod_{i=1}^{k}(x_{2k+2i-1}^2+x_{2k+2i}^2)-x_{4k+1}^{2k}
\end{equation*}
%(of degree $2k$ in $4k+1$ variables) 
while for $d$ odd, say $d=2k+1$, $k\geq 0$, we look at the number of zeros of the form
\begin{equation*} \label{e:one1o}
f({\bf x})=
f(x_1, \ldots, x_{4k+3})=
x_1\prod_{i=1}^{k}(x_{2i}^2+x_{2i+1}^2)+x_{2k+2}\prod_{i=1}^{k}(x_{2k+2i+1}^2+x_{2k+2i+2}^2)-x_{4k+3}^{2k+1}
\end{equation*}
%(of degree $2k+1$ in $4k+3$ variables, $k\geq 2$) 
in a box $1\leq x_i\leq P$.
%\indent One can also consider the problem of finding an asymptotic formula for the number of zeros of the form $f$ inside the box $|x_i|\leq P$ in case $d$ is odd, $d=2k+1$, but then the main term would come from the ``degenerate" zeros, i.e., the zeros that %lie inside the box and that 
%also lie in the set \begin{equation*} U=\{{\bf x}\in \mathbb{Z}^{4k+3}: f_d(x_1, \ldots, x_d)=f_d(x_{d+1}, \ldots, x_{2d})=x_{2d+1}%x_1\prod_{i=1}^{k}(x_{2i}^2+x_{2i+1}^2)=x_{2k+2}\prod_{i=1}^{k}(x_{2k+2i+1}^2+x_{2k+2i+2}^2)=x_{4k+3}= 0\} \end{equation*} (in particular from the lattice points in the $4k$-dimensional subspace $V=\{x_1=x_{2k+2}=x_{4k+3}=0\}$ of $\mathbb{Q}^{4k+3}$), their number being $2^{4k}P^{4k}+O(P^{4k-1}).$\\
When convenient we shall abbreviate and write
\begin{equation} \label{e:one1}
f({\bf x})=
f(x_1, \ldots, x_{2d+1})=
f_d(x_1, \ldots, x_d)+f_d(x_{d+1}, \ldots, x_{2d})-x_{2d+1}^d
\end{equation}
where
\begin{equation*}
f_d(x_1, \ldots, x_d)=\left\{ \begin{array}{ll}
\prod_{i=1}^{k}(x_{2i-1}^2+x_{2i}^2) & \mbox{ if $d$ is even, $d=2k$, $k\geq 1$}\\
x_1\prod_{i=1}^{k}(x_{2i}^2+x_{2i+1}^2) & \mbox{ if $d$ is odd, $d=2k+1$, $k\geq 0$}
\end{array}
\right..
\end{equation*}
We could define $f_d=f_d(x_1, \ldots, x_d)$ recursively: $f_1(x_1)=x_1$, $f_2(x_1, x_2)=x_1^2+x_2^2$, and $f_d(x_1, \ldots, x_d)=(x_{d-1}^2+x_{d}^2)\cdot f_{d-2}(x_1, \ldots, x_{d-2})$ for $d\geq 2.$\\
\indent Taking $x_{d-2}=x_{2d-2}=x_{d-4}=x_{2d-4}=\ldots=1$,\, $x_{d-3}=x_{2d-3}=x_{d-5}=x_{2d-5}=\ldots =0$ and applying the four square theorem we see that this form represents every integer infinitely many times; our interest lies in establishing the asymptotic formula for the number of representations of zero with $1\leq x_i\leq P$. Only minor changes are required to prove the expected asymptotic formula for the number of representations of a large positive integer $N$ by this form (or by the form obtained by changing the $-$ sign in the front of $x_{2d+1}^{d}$ to $+$) with the variables restricted so that $1\leq x_i \leq \lfloor N^{\frac{1}{d}} \rfloor $.\\
\indent Note that if $d\geq 3$ then the form $f$ is singular (e.g., consider the point (1, 0, 0, \ldots, 0)) (so Deligne's estimate for complete exponential sums is not applicable) and its $h$-invariant is 3. (The $h$-invariant of a form $F$, denoted $h(F)$, is defined as the smallest integer $h$ such that $F$ can be written in the form
$$F({\bf x})=A_1({\bf x})B_1({\bf x})+\ldots+A_h({\bf x})B_h({\bf x})$$
where $A_i({\bf x})$ and $B_i({\bf x})$ are forms of positive degrees having rational coefficients.)\\
\indent In what follows, $\varepsilon$ can take any positive real value in any statement in which it appears, $e(\alpha)=e^{2\pi i \alpha}$, the symbols $\ll$ and $O$ have their usual meanings with the implicit constants depending at most on $\varepsilon$ and $d$. If $A$ is a set, $\#A$ denotes the cardinality of $A.$ If $n, m$ are positive integers, $d(n)$ denotes the number of positive divisors of $n$ and $d_m(n)$ denotes the number of (ordered) factorizations of $n$ with $m$ positive integral factors. For a real number $x$, $\lfloor x \rfloor$ and $\lceil x \rceil$ denote the greatest integer less than or equal to $x$ and the smallest integer greater than or equal to $x$ respectively.
\begin{thm} \label{t:one1}
The number of zeros  $R(0;P)$ of the form $f$ in (\ref{e:one1}) with $1\leq x_i \leq P$ satisfies
\begin{equation*}
R(0;P)=P^{d+1}\mathfrak{S}J+O(P^{d+1-\frac{1}{1+5\cdot 2^{d-1}}+\varepsilon})
\end{equation*}
where $${\mathfrak{S}}=\sum_{q=1}^{\infty}\sum_{\substack{a=1\\(a,q)=1}}
^{q}
q^{-2d-1}S(q,a) \mbox{ with }S(q,a)=\sum_{{\bf z} \bmod q} e\left( \frac{a}{q}f({\bf z})\right),$$
and
$$J=\int_{-\infty}^{\infty}\left( \int_{[0,1]^{2d+1}}e(\gamma f({\bf \xi}))\,d{\bf
\xi} \right) d\gamma;$$
$\mathfrak{S}$ and $J$ are positive.
\end{thm}
\begin{rem}
One can also consider the problem of finding an asymptotic formula for the number of zeros of the form $f$ inside the box $|x_i|\leq P$ but then the contribution from those ``degenerate" zeros of $f$ that lie inside the box $|x_i|\leq P$ and also in the set
\begin{equation*}
U=\{{\bf x}\in \mathbb{Z}^{2d+1}: f_d(x_1, \ldots, x_d)=f_d(x_{d+1}, \ldots, x_{2d})=x_{2d+1}= 0\}
\end{equation*}
would be $2^{2d-2}P^{2d-2}+O_{d}(P^{2d-3})$ in case $d$ is odd and $2^{2d-6}d^2P^{2d-4}+O_{d}(P^{2d-5})$ in case $d$ is even and thus the contribution from the zeros of $f$ with at least one of the $x_i$ equal to zero would be $\gg P^{d+2}$ if $d\geq 5$ and cannot be determined precisely via the circle method. On the other hand, the number of zeros of $f$ with $|x_i|\leq P$ but none of the $x_i$ equal to zero is $2^{2d+1}$ times the number of zeros with $1\leq x_i\leq P$ in case $d$ is even, and is $2^{2d-1}$ times the sum of the number of zeros with $1\leq x_i\leq P$ of the three forms
$$f_d(x_1, \ldots, x_d)+f_d(x_{d+1}, \ldots, x_{2d})-x_{2d+1}^d,$$
$$-f_d(x_1, \ldots, x_d)+f_d(x_{d+1}, \ldots, x_{2d})-x_{2d+1}^d$$
and
$$f_d(x_1, \ldots, x_d)-f_d(x_{d+1}, \ldots, x_{2d})-x_{2d+1}^d;$$
theorem 1 applies to the latter two forms as well.
\end{rem}
\Section{The Minor Arcs}
\label{minorarcs}
\noindent Let $P$ be a large positive integer, $I=\{1, 2, \ldots, P\}$, $B=I^{2d+1}$ and
\begin{equation} \label{e:two1}
F(\alpha)= \sum_{{\bf x}\in B} e(\alpha f({\bf x})).
\end{equation}
Then the number of zeros of the form $f$ inside the box $1\leq x_i\leq P$ is
\begin{equation} \label{e:two2}
R(0;P)=\int_{0}^{1} F(\alpha)\,d\alpha=\int_{0}^{1} F_1(\alpha)^2 F_2(\alpha)\,d\alpha
\end{equation}
where
$$F_1(\alpha)=\sum_{x_1, \ldots, x_{d} \in I} e(\alpha f_d(x_1, \ldots, x_{d}))$$
and
$$F_2(\alpha)=\sum_{x \in I} e(-\alpha x^{d}).$$
\indent Let $\delta$ be a sufficiently small positive real number to be specified 
later. For $1 \leq q \leq P^{\delta},\; 1 \leq a \leq q,\; (a,q)=1,$ define the major arcs $\mathfrak{M}(q,a)$ by
\begin{equation*}
\mathfrak{M}(q,a)=\left\{\alpha:\left| \alpha
- \frac{a}{q}\right| \leq P^{\delta-d}\right\}
\end{equation*}
and let $\mathfrak{M}$ be the union of the $\mathfrak{M}(q,a)$. The set $\mathfrak{m}=(P^{\delta - d}, 1+P^{\delta - d}]\backslash \mathfrak{M}$ forms the minor arcs.
\begin{lem} \label{l:two1} We have
\begin{equation*}
\int_{0}^{1} \left| F_1(\alpha)\right|^2 d\alpha \ll P^{d+\varepsilon}.
\end{equation*}
\end{lem}
\begin{proof} We assume that $d$ is even, $d=2k$, $k\geq 1$; the proof is similar in case $d$ is odd. We have 
\begin{equation*}
\int_{0}^{1}|F_1(\alpha)|^2 d\alpha
=\sum_{0\leq n \leq 2^kP^{d}}a(n)^2
%=1+\sum_{1\leq n \leq 2^kP^{d}}a(n)^2
\end{equation*}
where \begin{equation*}
a(n) = \#\{(x_1, \ldots, x_{d})\in I^{d}: f_d(x_1, \ldots, x_{d})=n\} \mbox{ for }
n\geq 0.
\end{equation*}
Clearly $a(0)=1$ and $\sum_{0\leq n \leq 2^kP^{d}}a(n)=P^{d}$. For $n\geq 1$, in any solution of $\prod_{i=1}^{k}(x_{2i-1}^2+x_{2i}^2)=n$ we must have $x_{2i-1}^2+x_{2i}^2=n_i$ ($1\leq i\leq k$) for some factorization $n=n_1\ldots n_k$ of $n$ with $k$ (positive) factors. For a positive integer $r$, the number of solutions of the equation $s^2+t^2=r$ in non-negative integers  $s, t$ is $\leq 8d(r) \ll r^{\varepsilon}$. Thus, for $1 \leq n\leq 2^kP^{d}$ we have
\begin{equation*} a(n)\ll \sum_{n_1\ldots n_k=n}n_1^{\varepsilon}\ldots n_k^{\varepsilon}=d_{k}(n)n^{\varepsilon}\ll n^{\varepsilon}n^{\varepsilon}\ll P^{2d\varepsilon}.
\end{equation*}
Hence
\begin{equation*}
\int_{0}^{1} |F_1(\alpha)|^2 d\alpha = 1+\sum_{1\leq n \leq 2^kP^{d}}a(n)^2
\ll 1+P^{2d\varepsilon}\sum_{1\leq n \leq 2^kP^{d}}a(n)\ll P^{d+2d\varepsilon}.
\qedhere
\end{equation*}
\end{proof}
\begin{lem} \label{l:two2}
We have
\begin{equation*}
\int_{\mathfrak{m}}F(\alpha)\,d\alpha \ll P^{d+1+\varepsilon-\frac{\delta}{2^{d-1}}}.
\end{equation*}
\end{lem}
\begin{proof}
For $\alpha \in \mathfrak{m}$,
by Dirichlet's theorem on Diophantine approximation, there exist $a, q$ with
$(a,q) =1, q\leq P^{d-\delta}$ and $\left| \alpha
- \frac{a}{q}\right| \leq q^{-1}P^{\delta-d} < q^{-2}.$ Since
$\alpha \in \mathfrak{m}\, \subseteq (P^{\delta - d},
1-P^{\delta - d})$, we must have $1\leq a < q$ and $q
>P^{\delta}$ (otherwise $\alpha$ would be in $\mathfrak{M}(q,a)$). By Weyl's
inequality (See \cite{DBook} or \cite{VaughanBook}),
$$F_2(\alpha)\ll P^{1+\varepsilon-\frac{\delta}{2^{d-1}}}.$$
This, with (\ref{e:two2}) and lemma \ref{l:two1} gives
\begin{eqnarray*}
\left| \int_{\mathfrak{m}}F(\alpha)\,d\alpha \right|
&= & \left| \int_{\mathfrak{m}}F_1^2(\alpha) F_2(\alpha)\,d\alpha \right|
\leq \left( \sup_{\alpha\in\mathfrak{m}}|F_2(\alpha)|
\right) \left( \int_{\mathfrak{m}} |F_1(\alpha)|^2\,d\alpha \right)\\
&\leq & \left( \sup_{\alpha\in\mathfrak{m}}|F_2(\alpha)| \right)
\left( \int_{0}^{1} |F_1(\alpha)|^2 d\alpha \right)
\ll P^{d+1+\varepsilon-\frac{\delta}{2^{d-1}}}. \qedhere
\end{eqnarray*}
\end{proof}
\Section{The Major Arcs}
\label{majorarcs}
\noindent To deal with the major arcs we need to obtain an approximation to
the generating function $F(\alpha)$, (\ref{e:two1}), in terms of the auxiliary functions
\begin{equation} \label{e:three1}
S(q,a)=\sum_{{\bf z} \bmod q} e\left( \frac{a}{q}f({\bf z})\right)
\end{equation}
and 
\begin{equation} \label{e:three2}
I(\beta)=\int_{[0,P]^{2d+1}}e(\beta f({\bf \xi}))\,d{\bf \xi}.
\end{equation}
\begin{lem} \label{l:three1}
For $\alpha \in \mathfrak{M}(q,a),$ writing $\alpha = (a/q)+\beta,$ we have
\begin{equation*}
F(\alpha)=q^{-2d-1} S(q,a)I(\beta)+O(P^{2d+2\delta}).
\end{equation*}
\end{lem}
\begin{proof} We have
\begin{eqnarray*}
F\left( \frac{a}{q}+\beta \right) &=& \sum_{{\bf x}\in B}e\left( \left( \frac{a}{q}+\beta \right) f({\bf x})\right)\\
&=& \sum_{{\bf z} \bmod q} \ \sum_{{\bf y}:q{\bf y}+{\bf z}\in B}e\left( \left( \frac{a}{q}+\beta \right) f(q{\bf y}+{\bf z})\right)\\
&=& \sum_{{\bf z} \bmod q} \ \sum_{{\bf y}:q{\bf y}+{\bf z}\in B}e\left( \frac{a}{q}f(q{\bf y}+{\bf z})\right) e(\beta f(q{\bf y}+{\bf z}))\\
&=& q^{-4k-1}\sum_{{\bf z} \bmod q} e\left( \frac{a}{q}f({\bf z})\right) \sum_{{\bf
y}:q{\bf y}+{\bf z}\in B}q^{2d+1}e(\beta f(q{\bf y}+{\bf z})).
\end{eqnarray*}
The inner sum is a Riemann sum of the function $e(\beta f({\bf \xi}))$ over a $2d+1$-dimensional cuboid which depends on ${\bf z}$.
We want to replace it by the integral over the $2d+1$-dimensional cube
$[0,P]^{2d+1}$ %(which is $I(\beta)$) 
and estimate the error in doing so.
Firstly, any first-order partial derivative of $e(\beta f({\bf \xi}))$ is of the form $2\pi i\beta g({\bf \xi})e(\beta f({\bf \xi}))$
where $g({\bf \xi})$ is a degree $d-1$ form (in $2d+1$ real variables)
whose coefficients are bounded and which itself is, therefore, $O(P^{d-1})$ on the
cube $[0,P]^{2d+1}$.
Therefore, $e(\beta f({\bf \xi}))$ does not vary by more that $O(|\beta| q P^{d-1})$ on a cube of side $q$.
There are $q^{2d+1}$ terms in the outer sum and $\ll (P/q)^{2d+1}$ terms in each inner sum so the error in replacing the sum by the integral over the corresponding
cuboid is $O(|\beta| q P^{d-1} q^{2d+1}(P/q)^{2d+1})=O(|\beta|q P^{3d})$.
Secondly, each side of this cuboid depends on ${\bf z}$ but is within length $q$
of the sides of $[0,P]^{2d+1}$.
Since the integrand is bounded by 1 in absolute value, the error in replacing the region of integration by $[0,P]^{2d+1}$ is $O(q P^{2d})$.
Since we have $q\leq P^{\delta}$ and $|\beta|\leq P^{\delta-d}$, the total error is $O(P^{2d+2\delta})+O(P^{2d+\delta})=O(P^{2d+2\delta})$.
Hence,% from the definition of $I(\beta)$,
$$F\left( \frac{a}{q}+\beta \right)= q^{-2d-1}\sum_{{\bf z} \bmod q} e\left( \frac{a}{q}f({\bf z})\right) \left(I(\beta)+O(P^{2d+2\delta})\right)$$
which gives the lemma. \qedhere
\end{proof}
\begin{lem} \label{l:three2}
We have
$$\int_{\mathfrak{M}} F(\alpha)\,d\alpha=P^{d+1}\mathfrak{S}(P^{\delta})J(P^{\delta})+O(P^{d+5\delta})$$
where
\begin{equation} \label{e:three3}
{\mathfrak{S}}(Q)=\sum_{q\leq Q}\sum_{
\substack{
a=1\\
(a,q)=1}
}
^{q}
q^{-2d-1}S(q,a)
\end{equation}
and 
\begin{equation} \label{e:three4}
J(\mu)=\int_{|\gamma|<\mu}\left( \int_{[0,1]^{2d+1}}e(\gamma f({\bf \xi}))\,d{\bf
\xi} \right)d\gamma.
\end{equation}
\end{lem}
\begin{proof} Note that
${\mathfrak{M}}(q,a)=[(a/q)-P^{\delta-d},(a/q)+P^{\delta-d}]$ has length $2P^{\delta-d}$.
The change of variable \mbox{$\alpha =(a/q)+\beta$} and lemma \ref{l:three1} give
$$\int_{{\mathfrak{M}}(q,a)} F(\alpha)\,d\alpha=
 q^{-2d-1}S(q,a)\int_{-P^{\delta-d}}^{P^{\delta-d}}I(\beta)\,d\beta+O(P^{d+3\delta}).$$
Summing over $1 \leq q\leq P^{\delta},\; 1 \leq a \leq q,\; (a,q)=1,$
we get
$$\int_{\mathfrak{M}} F(\alpha) \,d\alpha=\mathfrak{S}(P^{\delta})\int_{-P^{\delta-d}}^{P^{\delta-d}}I(\beta)
\,d\beta+O(P^{d+5\delta}).$$
Applying the change of variable $\gamma=P^{d}\beta$, the integral in the above
equation becomes
$$P^{-d}\int_{-P^{\delta}}^{P^{\delta}}I(P^{-d}\gamma)\,d\gamma.$$
Since $f$ is a homogeneous function of degree $d$, the change of variable ${\bf \xi} '=P^{-1}{\bf \xi}$ gives
\begin{eqnarray*}
(P^{-d}\gamma)&=&\int_{[0,P]^{2d+1}}e(\gamma P^{-d} f({\bf \xi}))\,d{\bf \xi}
=\int_{[0,P]^{2d+1}}e(\gamma f({P^{-1}\bf \xi}))\,d{\bf \xi}\\
&=& P^{2d+1}\int_{{[0,1]^{2d+1}}}e(\gamma
 f({\bf \xi '}))\,d{\bf \xi '}
\end{eqnarray*}
which completes the proof of the lemma.
\end{proof}
\Section{The Singular Integral}
\label{singularintegral}
\noindent Note that the inner integral appearing in the definition of $J(\mu)$,
(\ref{e:three4}), is bounded trivially by 1.
Applying the change of variable ${\bf \xi'} = \gamma^{1/2k}{\bf \xi}$ we see that
$$\int_{[0,1]^{2d+1}}e(\gamma
f({\bf \xi}))\,d{\bf \xi}= \gamma^{-\frac{2d+1}{d}}\int_{\gamma^{\frac{1}{d}}[0,1]^{2d+1}}e(f({\bf
\xi'}))\,d{\bf \xi'}.$$
We think of the last integral as an iterated integral and apply Lemma 7.3 of
Vaughan \cite{VaughanBook} to each iteration to get the bound
\begin{equation} \label{e:four1}
\int_{[0,1]^{2d+1}}e(\gamma f({\bf \xi}))d{\bf \xi} \ll \mbox{min}
(1,\gamma^{-\frac{2d+1}{d}}).
\end{equation}
\begin{lem} \label{l:four1}
For $\mu\geq 1$,
\begin{equation*}
J(\mu) = J+O(\mu^{-\frac{d+1}{d}})
\end{equation*}
where
\begin{equation*}
J=\int_{-\infty}^{\infty}\left( \int_{[0,1]^{2d+1}}e(\gamma f({\bf \xi}))\,d{\bf
\xi} \right)d\gamma >0.
\end{equation*}
\end{lem}
\begin{proof}
The convergence of $J(\mu)$ to $J$ as $\mu \rightarrow \infty$ at the rate asserted follows from (\ref{e:four1}). To show that $J>0$, we imitate \cite[Chapter 4]{DBook}. Writing $\underline{\alpha}_1$ for $(\alpha_1, \ldots, \alpha_{d})$ and $\underline{\alpha}_2$ for $(\alpha_{d+1}, \ldots, \alpha_{2d})$,
\begin{equation*}
J=\lim_{\lambda\rightarrow\infty}\int_{-\lambda}^{\lambda}\left(\int_{[0,1]^{2d+1}}e(\gamma (f_d(\underline{\alpha}_1)+f_d(\underline{\alpha}_2)-\alpha_{2d+1}^{d})\,d\underline{\alpha}_1\,d\underline{\alpha}_2\,d\alpha_{2d+1}\right)d\gamma.
\end{equation*}
Since $\int_{0}^{1}e(-\gamma \alpha_{2d+1}^{d})\,d\alpha_{2d+1}=\int_{0}^{1}\frac{1}{d}\alpha^{\frac{1}{d}-1}e(-\gamma \alpha)\,d\alpha$ and $\int_{-\lambda}^{\lambda}e(\gamma \mu)\,d\gamma=\sin(2\pi\lambda\mu)/\pi\mu$ for any real numbers $\lambda$ and $\mu$,
\begin{equation*}
J=\frac{1}{d}\lim_{\lambda\rightarrow\infty}\int_{[0,1]^{2d+1}}\alpha^{\frac{1}{d}-1}\left(\frac{\sin (2\pi \lambda (f_d(\underline{\alpha}_1)+f_d(\underline{\alpha}_2)-\alpha))}{\pi (f_d(\underline{\alpha}_1)+f_d(\underline{\alpha}_2)-\alpha)}\right)d\underline{\alpha}_1\,d\underline{\alpha}_2\,d\alpha.
\end{equation*}
Putting $\beta=f_d(\underline{\alpha}_1)+f_d(\underline{\alpha}_2)-\alpha$ so that $\alpha=f_d(\underline{\alpha}_1)+f_d(\underline{\alpha}_2)-\beta$,
\begin{equation*}
J=\frac{1}{d}\lim_{\lambda\rightarrow\infty}\int_{-1}^{2^{\lfloor d/2 \rfloor +1}}\phi(\beta)\frac{\sin (2\pi \lambda \beta)}{\pi \beta}\,d\beta
\end{equation*}
where
\begin{equation*}
\phi(\beta)=\frac{1}{d}\int_{\{\beta <f_d(\underline{\alpha}_1)+f_d(\underline{\alpha}_2)<\beta+1\}\cap [0,1]^{2d}}
(f_d(\underline{\alpha}_1)+f_d(\underline{\alpha}_2)-\beta)^{\frac{1}{d}-1}\,d\underline{\alpha}_1\,d\underline{\alpha}_2.
\end{equation*}
Fourier's integral theorem for a finite interval is applicable and gives
\begin{equation*}
J=\phi(0)=\int_{\{0<f_d(\underline{\alpha}_1)+f_d(\underline{\alpha}_2)<1\}\cap [0,1]^{2d}}
(f_d(\underline{\alpha}_1)+f_d(\underline{\alpha}_2))^{\frac{1}{d}-1}\,d\underline{\alpha}_1\,d\underline{\alpha}_2.
\end{equation*}
Since the region of integration is a non-empty open set and the integrand is positive throughout the region of integration, $J>0$.
\end{proof}
\Section{The Singular Series}
\label{singularseries}
\noindent From the definition of ${\mathfrak{S}}(Q)$, (\ref{e:three3}), 
\begin{equation} \label{e:five1}
{\mathfrak{S}}(Q)=\sum_{q\leq Q}S(q)
%=\sum_{q\leq Q}\sum_{\substack{a=1\\(a,q)=1}}^{q} q^{-2d-1}S(q,a)
=\sum_{q\leq Q} q^{-2d-1}\sum_{\substack{a=1\\(a,q)=1}}^{q}S_1(q,a)^2S_2(q,a)
\end{equation}
where$$S_1(q,a)=\sum_{x_1, \ldots, x_{d}=1}^{q} e \left(\frac{a}{q}f_d(x_{1}, \ldots, x_{d})\right),\mbox{ }
S_2(q,a)=\sum_{x = 1}^{q} e\left( -\frac{a}{q}x^{d}\right)$$
and $$S(q)=q^{-2d-1}\sum_{\substack{a=1\\(a,q)=1}}^{q}S_1(q,a)^2S_2(q,a).$$
\noindent $S(q)$ is a multiplicative function of $q$.
In fact, if $(q_1,q_2)=(a,q_1q_2)=1$ and we choose $a_1,a_2$ so that $a_2q_1+a_1q_2=1$
then for $i=1,2
$\begin{equation} \label{e:five2}
S_i(q_1q_2,a)=S_i(q_1q_2,a_1q_2+a_2q_1)=S_i(q_1,a_1)S_i(q_2,a_2).
\end{equation}
This follows from the fact that as $r_1$ and $r_2$ run through
a complete (or reduced) set of residues modulo $q_1$ and $q_2$ respectively, $r_2q_1+r_1q_2$ runs through
a complete (or reduced) set of residues modulo $q_1q_2$.
Thus it suffices to study $S_i(q,a)$ when $q$ is a prime power and $(a,q)=1$. For a prime power $p^l$ and $p\nmid a$ we have
\begin{equation} \label{e:five3}
S_1(p^l,a)=\sum_{u=1}^{p^l}N_{d-2}(p,l,u)\left(\sum_{y=1}^{p^l}e\left(\frac{auy^2}{p^l}\right)\sum_{z=1}^{p^l}e\left(\frac{auz^2}{p^l}\right)\right).
\end{equation}
where
\begin{equation*}
N_{d-2}(p,l,u)=\#\{(x_1, \ldots, x_{d-2}): 1\leq x_1, \ldots, x_{d-2} \leq p^{l}, f_{d-2}(x_{1}, \ldots, x_{d-2})\equiv u \,(\mbox{mod }p^l)\}.
\end{equation*}
For a fixed $u$ with $p^j\|u$, writing $u= p^ju_1$ and dividing the sum over $y$ (or $z$) into $p^j$ subsums by grouping $p^{l-j}$ consecutive terms
together we get
\begin{equation*}
\sum_{y=1}^{p^l}e\left(\frac{auy^2}{p^l}\right)
=\sum_{z=1}^{p^l}e\left(\frac{auz^2}{p^l}\right)
=\sum_{v_1=1}^{p^j}\sum_{v_2=1}^{p^{l-j}}
e\left(\frac{au_1(p^{l-j}v_1+v_2)^2}{p^{l-j}}\right)=p^j\sum_{v_2=1}^{p^{l-j}}
e\left(\frac{au_1v_2^2}{p^{l-j}}\right).
\end{equation*}
The sum over $v_2$ is a quadratic Gauss sum modulo $p^{l-j}$ and its value is well known (e.g., see \cite[Chapter 9]{MontgomeryVaughan1}). We find that
\begin{equation*}
\sum_{y=1}^{p^l}e\left(\frac{auy^2}{p^l}\right)\sum_{z=1}^{p^l}e\left(\frac{auz^2}{p^l}\right)=\left\{\begin{array}{ll}
p^{2l} & \mbox{if } j=l\\
p^{l+j} & \mbox{if } p \equiv 1 \, (\mbox{mod }4)\\
(-1)^{l+j}p^{l+j} & \mbox{if } p \equiv 3 \, (\mbox{mod }4)\\
0 & \mbox{if } p=2,\,  j=l-1\\
2\cdot \sqrt{-1}\cdot 2^{l+j} & \mbox{if } p=2,\, j\leq l-2
\end{array}
=G(p,l,j)
\right.,
\end{equation*}
say, depends on $p, l$ and $j$ but not on $a$ or $u_1$. Thus, from (\ref{e:five3}),
\begin{equation} \label{e:five4}
S_1(p^l,a)=p^{2l}M_{d-2}(p,l,l)+\sum_{j=0}^{l-1}M^{*}_{d-2}(p,l,j)G(p,l,j)
\end{equation}
where, for $d\geq 3$, $l\geq 1$, $0\leq j \leq l,$
\begin{equation*}
M_{d-2}(p,l,j):=\#\{(x_1, \ldots, x_{d-2}): 1\leq x_1, \ldots, x_{d-2} \leq p^{l},\, p^j|f_{d-2}(x_{1}, \ldots, x_{d-2})\}.
\end{equation*}
and, for $d\geq 3$, $l\geq 1$, $0\leq j \leq l-1$,
\begin{eqnarray*}
M^{*}_{d-2}(p,l,j)&:=&\#\{(x_1, \ldots, x_{d-2}): 1\leq x_1, \ldots, x_{d-2} \leq p^{l},\, p^j\|f_{d-2}(x_{1}, \ldots, x_{d-2})\}\\
&=&M_{d-2}(p,l,j)-M_{d-2}(p,l,j+1).
\end{eqnarray*}
The equation (\ref{e:five4}) also holds for $d=1$ and $d=2$ if we define
\begin{equation*}
M_{-1}(p,l,j)=0 \mbox{ for } 0\leq j\leq l, \, M_{-1}^{*}(p,l,j)=0 \mbox{ for } 0\leq j\leq l-1,
\end{equation*}
\begin{equation*}
f_0=1,\, M_0(p,l,j)=\left\{
\begin{array}{ll}
1 \mbox{ if } j=0\\
0 \mbox{ if } 1\leq j\leq l
\end{array} \right.
\mbox{ and }
M_0^{*}(p,l,j)=\left\{
\begin{array}{ll}
1 \mbox{ if } j=0\\
0 \mbox{ if } 1\leq j\leq l-1
\end{array} \right..
\end{equation*}
It is clear that
\begin{equation*}
M_1(p,l,j)= p^{l-j} \mbox{ for } 0\leq j\leq l
\mbox{ and }
M_1^{*}(p,l,j)=p^{l-j}(1-p^{-1}) \mbox{ for } 0\leq j\leq l-1.
\end{equation*}
The treatment of the singular series is trivial in cases $d=1, 2, 3$ now; we dispose them off. %and deal with the cases $d\geq 4$ later.
When $d=1$ we have
\begin{equation} \label{e:fived1}
S_1(p^l,a)=0;\mbox{ } S_1(q,a)=S_2(q,a)=S(q)= 1 \mbox{ if } q=1, \, 0 \mbox{ if } q> 1;\mbox{ } {\mathfrak S}=1.
\end{equation}
When $d=2$ we have
\begin{equation} \label{e:fived2}
S_1(p^l,a)=G(p,l,0);\mbox{ } |S_1(q,a)|\leq 2q;\mbox{ }|S_2(q,a)|\leq \sqrt{2q};\mbox{ } |S(q)|\leq 4\sqrt{2}q^{-\frac{3}{2}}.
\end{equation}
When $d=3$ we have
\begin{equation} \label{e:fived3}
S_1(p^l,a)|< (l+1)p^{2l};\mbox{ } |S_1(q,a)|< q^2d(q);\mbox{ }|S_2(q,a)|\ll q^{2/3};\mbox{ } |S(q)|\ll q^{-\frac{4}{3}+\varepsilon}.
\end{equation}
%There is one value of $x_1$ modulo $p^l$ that is divisible by $p^l$ while for $0\leq j \leq l-1$ there are $p^{l-j}(1-p^{-1})$ values of $x_1$ modulo $p^l$ that are exactly divisible by $p^j$. Hence, for $k\geq 2$ and $0\leq j\leq l-1$ we have
%$$M^{*}_{2k-2}(p,l,j)=\sum_{\substack{0\leq j_i\leq l-1\\ j_1+\ldots +j_{k-1}=j}} \prod_{i=1}^{k-1}M^{*}_2(p,l,j_i)$$
%and
%$$M^{*}_{2k-1}(p,l,j)=\sum_{\substack{0\leq j_i\leq l-1\\ j_0+\ldots +j_{k-1}=j}} p^{l-j_{0}}(1-p^{-1})\prod_{i=1}^{k-1}M^{*}_1(p,l,j_i).$$
%Also, for $k\geq 2$,
%$$ M_{2k-2}(p,l,l)\leq \binom{k-1}{1} M_2(p,l,l)\cdot (p^{2l})^{k-2}+\sum_{\substack{0\leq j_i\leq l-1\\ j_1+\ldots +j_{k-1}\geq l}} \prod_{i=1}^{k-1}M^{*}_2(p,l,j_i)$$
%and
%\begin{eqnarray*} M^{*}_{2k-1}(p,l,l)& \leq & 1\cdot p^{(2k-2)l}+\binom{k-1}{1} M_2(p,l,l)\cdot p^{(2k-3)l}\\ & & +\sum_{\substack{0\leq j_i\leq l-1\\ j_0+\ldots +j_{k-1}\geq l}} p^{l-j_{0}}(1-p^{-1})\prod_{i=1}^{k-1}M^{*}_2(p,l,j_i).\end{eqnarray*}
\indent To deal with the cases $d\geq 4$, we first calculate $M_2(p,l,j)=\#\{(w,x): 1\leq w, x \leq p^{l},\, p^j|w^2+x^2\}$ for $1\leq j \leq l$. We do so by dividing the set $\{(w,x): 1\leq w, x \leq p^{l},\, p^j|w^2+x^2\}$ according to the highest power of $p$ dividing $w$. Thus, for $1\leq j \leq l$,
\begin{eqnarray*}
%&\#\{(w,x): 1\leq w, x \leq p^{l},\, p^j|w^2+x^2\}\\
M_2(p,l,j)&=&\#\{(w,x): 1\leq w, x \leq p^{l},\,p^{\lfloor \frac{j+1}{2} \rfloor}|w,\, p^j|w^2+x^2\}\\
&&+\sum_{m=0}^{\lfloor \frac{j-1}{2} \rfloor}
\#\{(w,x): 1\leq w, x \leq p^{l},\,p^m\|w,\, p^j|w^2+x^2\}\\
&=&p^{2l-2\lfloor \frac{j+1}{2} \rfloor}+\sum_{m=0}^{\lfloor \frac{j-1}{2} \rfloor}
\#\{(w',x'): 1\leq w', x' \leq p^{l-m},\,p\nmid w',\, p^{j-2m}|w'^2+x'^2\}.
\end{eqnarray*}
Since $-1$ has two square roots modulo $p^{j-2m}$ if $p\equiv 1 \pmod{4}$, one square root modulo $p^{j-2m}$ if $p^{j-2m}=2$, and no square root in other cases, the $m^{th}$ term in the sum equals $p^{l-m}(1-\frac{1}{p})\cdot 2\cdot p^{(l-m)-(j-2m)}=2p^{2l-j}(1-\frac{1}{p})$  if $p\equiv 1 \pmod{4}$, $2^{2l-j-1}$ if $p^{j-2m}=2$ and %$l=\frac{j-1}{2}$ 
zero otherwise. Thus, for $1\leq j \leq l$,
$$M_2(p,l,j)=
%\#\{(w,x): 1\leq w, x \leq p^{l},\, p^j|w^2+x^2\}=
\left\{\begin{array}{ll}
p^{2l-2\lfloor \frac{j+1}{2} \rfloor}+2 \lfloor \frac{j+1}{2} \rfloor p^{2l-j}(1-\frac{1}{p}) & \mbox{if } p \equiv 1 \, (\mbox{mod }4)\\
p^{2l-2\lfloor \frac{j+1}{2} \rfloor} & \mbox{if } p \equiv 3 \, (\mbox{mod }4)\\
2^{2l-j} & \mbox{if } p=2
\end{array}
\right..$$
Hence, for $0\leq j\leq l-1$,% $M^{*}_1(p,l,j)$ is equal to
$$M^{*}_2(p,l,j)=
%\#\{(w,x): 1\leq w, x \leq p^{l},\, p^j\|w^2+x^2\}=
\left\{\begin{array}{ll}
(j+1)p^{2l-j}(1-\frac{1}{p})^2 & \mbox{if } p \equiv 1 \, (\mbox{mod }4)\\
p^{2l-j}(1-\frac{1}{p^2}) & \mbox{if } p \equiv 3 \, (\mbox{mod }4) \mbox{ and $j$ is even} \\
0 & \mbox{if } p \equiv 3 \, (\mbox{mod }4) \mbox{ and $j$ is odd}\\
2^{2l-j-1} & \mbox{if } p=2
\end{array}
\right..$$
Thus for $0\leq j \leq l-1$ we have $M^{*}_2(2,l,j)=2^{2l-1-j}$, $M^{*}_2(p,l,j)<p^{2l-j}$ if $p \equiv 3(\mbox{mod }4)$, and $M^{*}_2(p,l,j)<(j+1)p^{2l-j}$ for all $p$. Also, $M_2(p,l,l)\leq p^{l}$ if $p=2$ or $p \equiv 3(\mbox{mod }4)$, and $M_2(p,l,l)<(l+1)p^{l}$ for all $p$. To estimate $S_1(q, a)$ we consider the cases of even $d$ and odd $d$ separately.\\
\indent Now suppose that $d$ is even and $d\geq 4$, say $d=2k$ and $k\geq 2$. Then, for $0\leq j \leq l-1$,
\begin{equation*}
M^{*}_{2k-2}(p,l,j)=\sum_{\substack{0\leq j_i\leq l-1\\ j_1+\ldots +j_{k-1}=j}} \prod_{i=1}^{k-1}M^{*}_2(p,l,j_i)
\end{equation*}
Thus
\begin{eqnarray*}
M^{*}_{2k-2}(2,l,j) &=& \sum_{\substack{0\leq j_i\leq l-1\\ j_1+\ldots +j_{k-1}=j}} \prod_{i=1}^{k-1}2^{2l-1-j_i}\\
&=& \sum_{\substack{0\leq j_i\leq l-1\\ j_1+\ldots +j_{k-1}=j}} 2^{(k-1)(2l-1)-j}\\
& \leq & l^{k-1}2^{(k-1)(2l-1)-j} \leq l^{k-1}2^{2kl-2l-1-j}
.\end{eqnarray*}
If $p \equiv 3(\mbox{mod }4)$, then
\begin{eqnarray*}
M^{*}_{2k-2}(p,l,j)&< &\sum_{\substack{0\leq j_i\leq l-1\\ j_1+\ldots +j_{k-1}=j}} \prod_{i=1}^{k-1}p^{2l-j_i}\\
&=& \sum_{\substack{0\leq j_i\leq l-1\\ j_1+\ldots +j_{k-1}=j}} p^{2(k-1)l-j}\\
& \leq & l^{k-1}p^{2(k-1)l-j}=l^{k-1}p^{2kl-2l-j}%<p^{2(k-1)l-j}d(p^l)^{2k-2}
.\end{eqnarray*}
If $p \equiv 1(\mbox{mod }4)$, then
\begin{eqnarray*}
M^{*}_{2k-2}(p,l,j)&<&\sum_{\substack{0\leq j_i\leq l-1\\ j_1+\ldots +j_{k-1}=j}} \prod_{i=1}^{k-1}(j_i+1)p^{2l-j_i}\\
&=& p^{2(k-1)l-j}\sum_{\substack{0\leq j_i\leq l-1\\ j_1+\ldots +j_{k-1}=j}} \prod_{i=1}^{k-1}(j_i+1)\\
& < & l^{k-1}p^{2(k-1)l-j}\sum_{\substack{0\leq j_i\leq l-1\\ j_1+\ldots +j_{k-1}=j}}1\\
& < & l^{k-1}l^{k-1}p^{2(k-1)l-j}=l^{2k-2}p^{2kl-2l-j}%<p^{2(k-1)l-j}d(p^l)^{2k-2}
.\end{eqnarray*}
%\begin{eqnarray*} M^{*}_{k-1}(p,l,j)&=&\sum_{\substack{0\leq j_i\leq l-1\\ j_1+\ldots +j_{k-1}=j}} \prod_{i=1}^{k-1}M^{*}_1(p,l,j_i)<\left(\sum_{\substack{0\leq j_i\leq l-1\\ j_1+\ldots +j_{k-1}=j}} \prod_{i=1}^{k-1}(j_i+1)\right)p^{2(k-1)l-j}\\ &<&\left(\sum_{\substack{0\leq j_i\leq l-1\\ j_1+\ldots +j_{k-1}=j}}1\right)l^{k-1}p^{2(k-1)l-j}<l^{k-1}l^{k-1}p^{2(k-1)l-j}=l^{2k-2}p^{2kl-2l-j}%<p^{2(k-1)l-j}d(p^l)^{2k-2} .\end{eqnarray*}
%(Incidentally, $\sum_{j_1+\ldots +j_{k-1}=j,\,j_i\geq 0} \prod_{i=1}^{k-1}(j_i+1)$ is the coefficient of $x^j$ in $(1-x)^{-2(k-1)}$ and hence can be seen to be in fact $\leq (\mbox{min}(2k-2,j))^{2k-3}$.)
Dividing the $(2k-2)$-tuples $(x_1, \ldots, x_{2k-2})$ with $p^l|f_{2k-2}(x_1, \ldots, x_{2k-2})$ into the ones with one of the factors $x_{2i-1}^2+x_{2i}^2$ divisible by $p^l$ and the ones with none of the factors $x_{2i-1}^2+x_{2i}^2$ divisible by $p^l$ we see that for any prime $p$
\begin{eqnarray*}
M_{2k-2}(p,l,l)&\leq &\binom{k-1}{1} M_2(p,l,l)\cdot (p^{2l})^{k-2}+\sum_{\substack{0\leq j_i\leq l-1\\ j_1+\ldots +j_{k-1}\geq l}} \prod_{i=1}^{k-1}M^{*}_2(p,l,j_i)\\
& < &(k-1)(l+1)p^l\cdot p^{2kl-4l}+l^{2k-2}p^{2kl-2l-l}\\
& < &(l+1)^{2k-2}p^{2kl-3l}
.\end{eqnarray*}
(The last inequality follows since $(k-1)(l+1)+l^{2k-2}\leq (k-1)(2l)+l^{2k-2}=\binom{2k-2}{1}\cdot l+l^{2k-2}<(l+1)^{2k-2}.$)
Thus from (\ref{e:five4}),
\begin{equation*}
|S_1(p^l,a)|<(l+1)(l+1)^{2k-2}p^{2kl-l}=(p^l)^{2k-1}d(p^l)^{2k-1}.
\end{equation*}
Thus from (\ref{e:five2}) for even $d\geq 4$ and $(a,q)=1,$
\begin{equation} \label{e:five6}
|S_1(q,a)|\leq q^{2k-1}d(q)^{2k-1} \ll q^{2k-1+\varepsilon}=q^{d-1+\varepsilon}.
\end{equation}
\indent Now suppose that $d$ is odd and $d\geq 5$, say $d=2k+1$ and $k\geq 2$. Note that there is one value of $x_1$ modulo $p^l$ that is divisible by $p^l$ while for $0\leq j \leq l-1$ there are $p^{l-j}(1-p^{-1})$ values of $x_1$ modulo $p^l$ that are exactly divisible by $p^j$. Hence, for $0\leq j \leq l-1$,
\begin{equation*}
M^{*}_{2k-1}(p,l,j)=\sum_{\substack{0\leq j_i\leq l-1\\ j_0+\ldots +j_{k-1}=j}} p^{l-j_{0}}(1-p^{-1})\prod_{i=1}^{k-1}M^{*}_2(p,l,j_i).
\end{equation*}
Thus
\begin{eqnarray*}
M^{*}_{2k-1}(2,l,j) &=& \sum_{\substack{0\leq j_i\leq l-1\\ j_0+\ldots +j_{k-1}=j}} 2^{l-j_{0}}(1-2^{-1})\prod_{i=1}^{k-1}2^{2l-1-j_i}\\
&=& \sum_{\substack{0\leq j_i\leq l-1\\ j_0+\ldots +j_{k-1}=j}} \frac{1}{2}\cdot 2^{l+(k-1)(2l-1)-j}\\
& \leq & l^{k}2^{l+2kl-2l-(k-1)-j-1} \leq l^{k}2^{2kl-l-j-2}
.\end{eqnarray*}
If $p \equiv 3(\mbox{mod }4)$, then
\begin{eqnarray*}
M^{*}_{2k-1}(p,l,j)&<&\sum_{\substack{0\leq j_i\leq l-1\\ j_0+\ldots +j_{k-1}=j}} p^{l-j_{0}}(1-p^{-1})\prod_{i=1}^{k-1}p^{2l-j_i}\\
&=&\sum_{\substack{0\leq j_i\leq l-1\\ j_0+\ldots +j_{k-1}=j}} (1-p^{-1})p^{(2k-1)l-j}\\
&<&l^{k}p^{(2k-1)l-j}=l^{k}p^{2kl-l-j}%<p^{2(k-1)l-j}d(p^l)^{2k-2}
.\end{eqnarray*}
If $p \equiv 1(\mbox{mod }4)$, then
\begin{eqnarray*}
M^{*}_{2k-1}(p,l,j)& < &\sum_{\substack{0\leq j_i\leq l-1\\ j_0+\ldots +j_{k-1}=j}} p^{l-j_{0}}(1-p^{-1})\prod_{i=1}^{k-1}(j_i+1)p^{2l-j_i}\\
&<&(1-p^{-1})p^{(2k-1)l-j}\sum_{\substack{0\leq j_i\leq l-1\\ j_0+\ldots +j_{k-1}=j}} \prod_{i=1}^{k-1}(j_i+1)\\
&<&l^{k-1}p^{(2k-1)l-j}\sum_{\substack{0\leq j_i\leq l-1\\ j_0+\ldots +j_{k-1}=j}}1\\ &<&l^{k}l^{k-1}p^{(2k-1)l-j}=l^{2k-1}p^{2kl-l-j}%<p^{2(k-1)l-j}d(p^l)^{2k-2}
.\end{eqnarray*}
%(Incidentally, $\sum_{j_1+\ldots +j_{k-1}=j,\,j_i\geq 0} \prod_{i=1}^{k-1}(j_i+1)$ is the coefficient of $x^j$ in $(1-x)^{-2(k-1)}$ and hence can be seen to be in fact $\leq (\mbox{min}(2k-2,j))^{2k-3}$.)
Dividing the $(2k-1)$-tuples $(x_1, \ldots, x_{2k-1})$ with $p^l|f_{2k-2}(x_1, \ldots, x_{2k-1})$ into the ones with $x_1$ divisible by $p^l$, the ones with one of the factors $x_{2i}^2+x_{2i+1}^2$ divisible by $p^l$ and the ones with none of the factors of $f_{2k-1}$ divisible by $p^l$, we see that for any prime $p$
\begin{eqnarray*}
M_{2k-1}(p,l,l)&\leq & 1\cdot p^{(2k-2)l}+\binom{k-1}{1} M_2(p,l,l)\cdot p^{(2k-3)l}\\
& &+\sum_{\substack{0\leq j_i\leq l-1\\ j_0+\ldots +j_{k-1}\geq l}} p^{l-j_{0}}(1-p^{-1})\prod_{i=1}^{k-1}M^{*}_2(p,l,j_i)\\
&<&p^{2kl-2l}+(k-1)(l+1)\cdot p^{2kl-2l}+l^{2k-1}p^{2kl-l-l}\\
&<& (l+1)^{2k-1}p^{2kl-2l}
.\end{eqnarray*}
(The last inequality follows since $1+(k-1)(l+1)+l^{2k-2}\leq 1+(k-1)(2l)+l^{2k-2}<1+\binom{2k-1}{1}\cdot l+l^{2k-1}<(l+1)^{2k-1}.$)
Thus from (\ref{e:five4}),
\begin{equation*}
|S_1(p^l,a)|<(l+1)(l+1)^{2k-1}p^{2kl}=(p^l)^{2k}d(p^l)^{2k}.
\end{equation*}
Thus from (\ref{e:five2}) for odd $d\geq 5$ and $(a,q)=1,$
\begin{equation} \label{e:five7}
|S_1(q,a)|\leq q^{2k}d(q)^{2k} \ll q^{2k+\varepsilon}=q^{d-1+\varepsilon}.
\end{equation}\\
\indent It is well known from the study of Waring's problem that for $(a,q)=1$ (See \cite{DBook} or \cite{VaughanBook}),
\begin{equation*}
|S_2(q,a)|\ll q^{1-\frac{1}{d}}.
\end{equation*}
Thus, from (\ref{e:fived1})-(\ref{e:five7}) and (\ref{e:five1}),
\begin{equation}  \label{e:five5}
|S(q)|\ll q^{-1-\frac{1}{d}+\varepsilon}.
\end{equation}
\begin{lem} \label{l:five1}
The series \begin{equation*}
{\mathfrak{S}}=\sum_{q=1}^{\infty}S(q)
\end{equation*} converges
absolutely, ${\mathfrak{S}}>0$, and for any $Q\geq 1$ we have
\begin{equation*}
{\mathfrak{S}}(Q)= {\mathfrak{S}}+O(Q^{-\frac{1}{d}+\varepsilon}).
\end{equation*}
\end{lem}
\begin{proof} The convergence at the asserted rate follows from (\ref{e:five5}). The argument in \cite[Chapter 17]{DBook} applies to forms of any degree and shows that the positivity of ${\mathfrak{S}}$ would follow from the existence of a non-singular zero of the form $f.$ The point $(x_1, \ldots, x_{2d+1})$ with $x_1=x_3=\ldots =x_{2 \lfloor (d-1)/2 \rfloor+1}=x_{2d+1}=1,$ and the rest of the $x_i$ equal to zero is a non-singular zero of $f$.
\end{proof}
\Section{Completion of the Proof of Theorem 1}
\label{completion}
\noindent From (\ref{e:two2}) and lemmas \ref{l:two2}, \ref{l:three2}, \ref{l:four1}, \ref{l:five1} we get
\begin{eqnarray*}
R(0;P)&=&P^{d+1}(\mathfrak{S}+O(P^{-\frac{\delta}{d}+\varepsilon}))(J+O(P^{-\frac{d+1}{d}\delta}))+O(P^{d+5\delta})+O(P^{d+1+\varepsilon-\frac{\delta}{2^{d-1}}})\\
&=&P^{d+1}\mathfrak{S}J+O(P^{d+1+\varepsilon-\frac{\delta}{2^{d-1}}}+P^{d+5\delta}).
\end{eqnarray*}
Taking $\delta=\frac{2^{d-1}}{1+5\cdot 2^{d-1}}$ completes the proof of Theorem 1.\\\\
\begin{small} {\bf Acknowledgments.} I would like to thank Prof. Robert C. Vaughan for suggesting me the case $d=4$ of this problem and Sean Prendiville for asking me if I could do something similar for all $d$.\end{small}
\\\\


\begin{thebibliography}{99}
\bibitem{BDL} Birch, B.J., Davenport, H., Lewis, D.J.: The addition of norm
forms. Mathematika {\bf 9}, 75-82 (1962)
%\bibitem{Browning} Browning, T.D.: An overview of Manin's conjecture for delPezzo surfaces. In: Analytic Number Theory, pp. 39-55. Clay Math. Proc., {\bf 7}, Amer. Math. Soc., Providence (2007)
\bibitem{DBook} Davenport, H.: Analytic Methods for Diophantine Equations and
Diophantine Inequalities, 2nd edn.,
Cambridge University Press, Cambridge (2005)
%\bibitem{Manin} Franke, J., Manin, Yu. I., Tschinkel, Yu.: Rational points of bounded height on Fano varieties. Invent. Math. {\bf 95}(2), 421-435 (1989)
\bibitem{MontgomeryVaughan1} Montgomery, H.L., Vaughan, R.C.: Multiplicative Number Theory I. Classical Theory. Cambridge Studies in Advanced Mathematics, vol. 97. Cambridge University Press, Cambridge (2007)
\bibitem{Vaughan1} Vaughan, R.C.: On Waring's problem for cubes. J. Reine
Angew. Math. {\bf 365}, 122-170 (1986)
\bibitem{Vaughan3} Vaughan, R.C.: A new iterative method in Waring's problem. Acta Math. {\bf 162}, 1-71 (1989)
\bibitem{VaughanBook} Vaughan, R.C.: The Hardy-Littlewood Method, 2nd edn.
Cambridge Tracts in Mathematics, vol. 125. Cambridge University Press, Cambridge (1997)
%\bibitem{c7i} Verma, M.: Representation of integers by a family of cubic forms. Ramanujan Journal (to appear, available online)% {\bf 00}, 000-000 (2014?)
%\bibitem{c7ii} Verma, M.: Representation of integers by a family of cubic forms II. Submitted to Ramanujan Journal.% {\bf 00}, 000-000 (2014?)
\end{thebibliography}
\end{document}